\documentclass{amsart}
\usepackage[utf8]{inputenc}
\usepackage{amsmath,amsthm,amssymb,graphicx,enumerate,mathtools}
\usepackage{xcolor}
\title{A probabilistic analysis of the Neumann series iteration}
\makeatletter
\newcommand\thankssymb[1]{\textsuperscript{\@fnsymbol{#1}}}
\makeatother
\author[Yiting Zhang]{Yiting Zhang\thankssymb{1}}
\author[Thomas Trogdon]{Thomas Trogdon\thankssymb{2}}
\thanks{\thankssymb{1} University of California, Irvine. Email: zjx687691@gmail.com}
\thanks{\thankssymb{2} Department of Applied Mathematics, University of Washington. Email: trogdon@uw.edu.  Supported in part by NSF DMS-1753185, NSF DMS-1945652}
\usepackage{cooltooltips}
\usepackage[english]{babel}
\usepackage{caption}
\usepackage[toc,page]{appendix}
\usepackage{subfigure}
\newtheorem{lemma}{Lemma}[section]
\newtheorem{theorem}{Theorem}[section]
\newtheorem{proposition}{Proposition}[section]

\theoremstyle{definition}
\newtheorem{definition}{Definition}[section]

\theoremstyle{remark}
\newtheorem{remark}{Remark}
\graphicspath{ {./images/} }
\renewcommand{\vec}{\mathbf}
\renewcommand{\complement}{\mathrm{c}}

\usepackage{xargs}\newcommandx{\unsure}[2][1=]{\todo[linecolor=red,backgroundcolor=red!25,bordercolor=red,#1]{#2}}
\newcommandx{\change}[2][1=]{\todo[linecolor=blue,backgroundcolor=blue!25,bordercolor=blue,#1]{#2}}
\newcommandx{\info}[2][1=]{\todo[linecolor=OliveGreen,backgroundcolor=OliveGreen!25,bordercolor=OliveGreen,#1]{#2}}
\newcommandx{\improvement}[2][1=]{\todo[linecolor=red,backgroundcolor=gray!25,bordercolor=gray,#1]{#2}}
\newcommandx{\thiswillnotshow}[2][1=]{\todo[disable,#1]{#2}}
\usepackage{matlab-prettifier}
\begin{document}

\maketitle

\begin{abstract}
    Given a random matrix $A$ with eigenvalues between $-1$ and $1$, we analyze the number of iterations needed to solve the linear equation $(I-A)\vec x=\vec b$ with the Neumann series iteration. We give sufficient conditions for convergence of an upper bound of the iteration count in distribution. Specifically, our results show that when the scaled extreme eigenvalues of $A$ converge in distribution, this scaled upper bound on the number of iterations will converge to the reciprocal of the limiting distribution of the largest eigenvalue.
\end{abstract}
\section{Introduction}
Neumann series was introduced by Carl Neumann in 1877 in the context of potential theory \cite{neumann1877untersuchungen}. Neumann
series, or the more advanced Liouville-Neumann series has been applied to solve Fredholm integral equations \cite{tricomi1985integral}. In fact, apart from the theoretical applications of Neumann series, it plays an important role in solving computational problems. The Neumann series iteration, $\vec x^{(k)}=A\vec x^{(k-1)}+\vec b$, follows naturally from the actual Neumann series, i.e., $(I-A)^{-1}=\sum_{i=0}^{\infty}A^{k}$ when solving $(I-A)\vec x=\vec b$.

In this paper, we find that if the eigenvalues of an $n\times n$ symmetric matrix $A$ fall between $-1$ and $1$ and if the scaled extreme eigenvalues of $A$ converge in distribution as $n\to \infty$, then after scaling, a scaled upper bound on the number of iterations needed to solve $(I-A)\vec x=\vec b$ with the Neumann series iteration will converge to the reciprocal of the limiting distribution of the largest eigenvalue.

This provides the first step in the full probabilistic analysis of the Neumann series iteration.  In particular, our results show that a reasonably sharp upper bound depends only on the (rescaled) extreme eigenvalues as the matrix size tends to infinity.  The limiting distributions of these eigenvalues are often universal --- they are independent of distributional details of matrix entries\footnote{This can be true, within a class of distributions.}.  Therefore, one expects the convergence rate of the Neumann series to inherit this universality.  This phenomenon has been observed in many algorithms \cite{DiagonalRMT,Deift2014,Deift2017c} and rigorously established for eigenvalue algorithms \cite{Deift2016,Deift2017}.

This paper unfolds as follows. In Section 2 we introduce the algorithm and the halting criterion. In Section 3 we state the main theorem and give two examples where it applies. The proof of the main theorem is given in Section 4.
\section{The algorithm}
\subsection{The Neumann series iteration}
We first define Neumann series.
\begin{definition}
For $A\in \mathbb R^{n\times n}$, the Neumann series is defined formally as:
\begin{align*}
    \sum_{i=0}^\infty A^{i}=I+A^{1}+A^{2}+\cdots.
\end{align*}
\end{definition}
According to the above definition, we are interested in sufficient conditions for the Neumann series to converge. The following lemma and theorem provides the key for this study.
\begin{lemma}\label{lem:neumann}
If the spectral norm\footnote{$\|\cdot\|$ here denotes the matrix norm induced by the $\ell^{2}$-norm. In fact, this lemma can be generalized to any sub-multiplicative norm.} of $A$ satisfies $\|A\|<1$, then $(I-A)^{-1}$ exists, and
\begin{align*}
    (I-A)^{-1}=I+A+A^{2}+\cdots=\sum_{i=0}^\infty A^{i}.
\end{align*}
\end{lemma}
\begin{theorem}\label{t:neumann}
Given $A\in \mathbb R^{n\times n}$ with $\|A\|<1$ and $\vec b\in \mathbb {R}^{n}$, the numerical solution of $(I-A)\vec x=\vec b$ is found by applying the Neumann series iteration:
\begin{align*}
    \vec x_0 &= \vec 0,\\
    \vec x_k &= A\vec x_{k-1} + \vec b\\
    &=(I+A+A^{2}+\cdots+A^{k-1})\vec b\\
    &=\sum_{i=0}^{k-1}A^{i} \vec b, \quad k = 1,2,3\ldots.
\end{align*}
Here $\vec x_{k}$ converges to $\vec x=(I-A)^{-1}\vec b$ as $k\to \infty$.
\end{theorem}
The proof of the above lemma and theorem can be found in \cite[p.~457]{Burden2011}. The Neumann series iteration is the numerical algorithm we will use throughout this paper.
\subsection{Halting criterion}
The asymptotic behavior of $\vec x_{k}$ is well known by Theorem~\ref{t:neumann}, but we are more interested in the non-asymptotic case. Given a halting criterion, we are interested in the number of iterations needed to achieve that criterion. Two natural halting times are defined as follows:
\begin{definition}
Given $\epsilon>0$, define $k_{\epsilon}(A,\vec b)$ and $k_{\epsilon}^{*}(A,\vec b)$ by
\begin{align*}
    k_{\epsilon}(A,\vec b)&=\min \{k: \left\|\vec x-\vec x_{k}\right\|<\epsilon\},\\
    k_{\epsilon}^{*}(A,\vec b)&=\min \{k:\left\|(I-A)\vec x_{k}-\vec b\right\|<\epsilon\}.
\end{align*}
\end{definition}
For simplicity, our results will only concern upper bounds for the halting times:
\begin{definition}\label{d:K}
Given $\epsilon>0$, we define $K_{\epsilon}(A,\vec b)$ and $K_{\epsilon}^{*}(A,\vec b)$ to be
\begin{align*}
    K_{\epsilon}(A,\vec b)&=\min \left\{k: \left\|\sum_{i=k}^{\infty}A^{i}\right\|\left\|\vec b\right\|<\epsilon\right\},\\
    K_{\epsilon}^{*}(A,\vec b)&=\min \left\{k: \left\|(I-A)\sum_{i=k}^{\infty}A^{i}\right\|\left\|\vec b\right\|<\epsilon\right\}.
\end{align*}
\end{definition}
\begin{proposition}
$k_{\epsilon}(A,\vec b)\leq K_{\epsilon}(A,\vec b)$ and $k_{\epsilon}^{*}(A,\vec b)\leq K_{\epsilon}^{*}(A,\vec b)$.
\end{proposition}
\begin{proof}
For the first inequality, we show that 
\begin{align*}
    \left\|\vec x-\vec x_{k}\right\|\leq \left\|\sum_{i=k}^{\infty}A^{i}\right\|\left\|\vec b\right\|.
\end{align*} 
The exact solution of $(I-A)\vec x=\vec b$ is $\vec x=(I-A)^{-1}\vec b$. Based on Lemma~\ref{lem:neumann}, $\vec x=\sum_{i=0}^{\infty} A^{i}\vec b$. By Theorem~\ref{t:neumann}, we also know that $\vec x_k=\sum_{i=0}^{k-1}A^{i}\vec b$. Therefore, we have
\begin{align*}
    \left\|\vec x- \vec x_k\right\|=\left\|\sum_{i=0}^{\infty} A^{i}\vec b-\sum_{i=0}^{k-1} A^{i}\vec b\right\|=\left\|\sum_{i=k}^{\infty} A^{i}\vec b\right\|\leq \left\|\sum_{i=k}^{\infty} A^{i}\right\|\|\vec b\|.
\end{align*}
Similarly, for the second inequality, we show that 
\begin{align*}
    \left\|(I-A)\vec x_{k}-\vec b\right\|\leq \left\|(I-A)\sum_{i=k}^{\infty}A^{i}\right\|\left\|\vec b\right\|.
\end{align*}
If $\vec x$ is the exact solution, we have
\begin{align*}
    \left\|(I-A) \vec x_k - \vec b\right\|&=\left\|(I-A) \vec x_k - \vec b-[(I-A)\vec x-\vec b]\right\|\\
    &=\left\|(I-A) (\vec x_k - \vec x)\right\|\\
    &=\left\|(I-A) \sum_{i=k}^{\infty}A^{i}\vec b\right\|\\
    &\leq \left\|(I-A)\sum_{i=k}^{\infty}A^{i}\right\|\left\|\vec b\right\|.
\end{align*}
Thus, $k_{\epsilon}(A,\vec b)\leq K_{\epsilon}(A,\vec b)$ and $k_{\epsilon}^{*}(A,\vec b)\leq K_{\epsilon}^{*}(A,\vec b)$.
\end{proof}
Now, to show that the upper bounds $K_\epsilon$ and $K^*_\epsilon$ are sharp we give a sufficient condition for equality to hold. Suppose $\lambda$ is the largest eigenvalue of $A$, and $0<\lambda<1$. If $\vec b$ is the eigenvector of $A$ which corresponds to $\lambda$, we have $k_{\epsilon}(A,\vec b)=K_{\epsilon}(A,\vec b)$. This can be verified by showing that $\|\vec x-\vec x_{k}\|=\left\|\sum_{i=k}^{\infty} A^{i}\right\|\|\vec b\|$:
\begin{align*}
    \|\vec x- \vec x_k\|=\left\|\sum_{i=k}^{\infty} A^{i}\vec b\right\|=\left\|\sum_{i=k}^{\infty} \lambda^{i}\vec b\right\|=\sum_{i=k}^{\infty} \lambda^{i}\|\vec b\|=\left\|\sum_{i=k}^{\infty}A^{i}\right\|\|\vec b\|.
\end{align*}
Suppose the largest eigenvalue of $(I-A)$ is $\mu$ and $0<\mu<1$. If $\vec b$ is the eigenvector of $(I-A)$ which corresponds to $\mu$, we have $k_{\epsilon}^{*}(A,\vec b)= K_{\epsilon}^{*}(A,\vec b)$. The verification is similar as before.
\section{Results}
In this section we first state the main theorem and then provide two examples where it applies.
\subsection{Main theorem}
\begin{definition}
A random variable $X_{n}$ converges in distribution to $X$ as $n\to \infty$ if
\begin{align*}
    F_{X_n}(t)=\mathbb {P}(X_{n}\leq t)\to \mathbb {P}(X\leq t)=F_{X}(t)
\end{align*}
as $n \to \infty$ at every $t$ where $F_{X}(t)$ is continuous. Here $F_{X_n}(t)$ and $F_{X}(t)$ are the cumulative distribution functions of $X_{n}$ and $X$, respectively.
\end{definition}
\begin{theorem}\label{t:main}
Suppose $(M_{n})_{n\geq 1}$, $M_{n}\in \mathbb {R}^{n\times n}$ (or $\mathbb {C}^{n\times n}$) is a sequence of symmetric (or Hermitian) random matrices with eigenvalues:
\begin{align*}
    -1<\lambda_{1}\leq \lambda_{2}\leq \cdots\leq \lambda_{n}<1,\quad \lambda_{j}=\lambda_{j}(n).
\end{align*}
Suppose for some $\alpha \geq \beta >0$, we have
\begin{align*}
    n^{\alpha}(1-\lambda_{n})&\to X\text{ in distribution as }n\to \infty,\\
    n^{\beta}(\lambda_{1}+1)&\to Y\text{ in distribution as }n\to \infty,
\end{align*}
where both $F_{X}$ and $F_{Y}$ are continuous and supported on $[0,\infty)$. Let $\vec b$ be a unit vector and fix  $0<\epsilon < 1/2$. Then
\begin{align*}
    \frac {K_{\epsilon}(M_{n}, \vec b)}{\alpha \log (n/\epsilon^{1/\alpha})n^{\alpha}}\to \frac {1}{X}
\end{align*}
in distribution as $n\to \infty$.
\end{theorem}
\begin{remark}
Although Theorem~\ref{t:main} is about $K_{\epsilon}(M_{n},\vec b)$, we can state a similar theorem for $K^{*}_{\epsilon}(M_{n},\vec b)$. Under the setting of Theorem~\ref{t:main}, for $\alpha\geq \beta>0$, it is reasonable to conjecture
\begin{align*}
    -\frac {K^{*}_{\epsilon}(M_{n}, \vec b)}{\log (\epsilon)n^{\alpha}}\to \max\left\{\frac {1}{X},\frac {1}{Y}\right\}
\end{align*}
in distribution as $n\to \infty$. 
\end{remark}
\subsection{Numerical verification}
\subsubsection{Independent and identially distributed eigenvalues}\label{ex:1}

Let $B$ be an $n\times n$ matrix with independent and identically distributed standard normal entries. Construct an $n\times n$ matrix $A_n$ by
\begin{align*}
    A_n = Q \Lambda Q^T,
\end{align*}
where $Q$ is found by applying the QR factorization to $B$ and
\begin{align*}
    \Lambda = \mathrm{diag}(\lambda_1,\lambda_2,\ldots,\lambda_n),
\end{align*}
where $(\lambda_i)_{i=1}^n$ is a collection of independent and identically distributed (iid) random variables and is uniform on $[-1,1]$. According to \cite{2006math.ph...9050M}, $Q$ is called a Haar orthogonal matrix. With this choice of $A_n$, it follows that $\|A_n\|<1$ almost surely. Therefore, the iteration in Theorem~\ref{t:neumann} converges with probability $1$.
Define
\begin{align*}
     \lambda_{\max,n} &= \max_{1 \leq i \leq n} \lambda_i,\\
    \lambda_{\min,n} &= \min_{1 \leq i \leq n} \lambda_i.
\end{align*}
\begin{definition}
Define $\mathrm{exp}\left(\lambda\right)$ to be the exponential distribution with parameter $\lambda$. The probability density function for a random variable with distribution $\mathrm{exp}\left(\lambda\right)$ is
\begin{align*}
    f(x;\lambda)=\begin{cases} \lambda e^{-\lambda x} & x \geq 0,\\ 0 & x < 0. \end{cases}
\end{align*}
\end{definition}
\begin{proposition}\label{p:m}
Both $n(1-\lambda_{\max,n})$ and $n(1+\lambda_{\min,n})$ converge in distribution to $\mathrm{exp}\left(1/2\right)$ as $n\to \infty$.
\end{proposition}
\begin{proof}
We only show that $n(1-\lambda_{\max ,n})\to\mathrm{exp}\left(1/2\right)$ in distribution as $n\to \infty$. The proof that $n(1+\lambda_{\min,n})\to\mathrm{exp}\left(1/2\right)$ in distribution as $n\to \infty$ follows similarly. When $\lambda \geq 0$, we have
\begin{align*}
    \mathbb P(\lambda_{\max ,n}\leq \lambda)&=\mathbb P(\lambda_1\leq \lambda,\lambda_2\leq \lambda,\ldots,\lambda_n\leq \lambda)\\
    &=\prod_{i=1}^n \mathbb P(\lambda_i\leq \lambda)\\
    &=\left(\frac {\lambda+1}{2}\right)^n.
\end{align*}
Define $\Lambda_{\max, n}:=n(1-\lambda_{\max,n})$, we have
\begin{align*}
    \mathbb P\left(\Lambda_{\max,n}\leq \lambda\right)&=\mathbb P\left(n(1-\lambda_{\max,n})\leq \lambda\right)\\
    &=1-\mathbb P\left(\lambda_{\max,n}<1-\frac {\lambda}{n}\right)\\
    &=1-\left(1+\frac {-\frac {\lambda}{2}}{n}\right)^n.
\end{align*}
Since $e^x=\lim_{n \to \infty}\left(1+x/n\right)^n$, $\lim_{n \to \infty}\mathbb P(\Lambda_{\max,n}\leq \lambda)=1-e^{-\lambda/2}$. When $\lambda<0$, by a similar argument, $\lim_{n \to \infty}\mathbb P(\Lambda_{\max,n}\leq \lambda)=0$. Therefore, $n(1-\lambda_{\max ,n})\to\mathrm{exp}\left(1/2\right)$ in distribution as $n\to \infty$.
\end{proof}
By Theorem~\ref{t:main}, with $\alpha=\beta=1$ and $X,Y\sim \mathrm{exp}(1/2)$, we have
\begin{align*}
    \frac {K_{\epsilon}(A_n,\vec b)}{n \log (n/\epsilon)}\to \frac {1}{X}
\end{align*}
in distribution as $n\to \infty$. Fix $\epsilon=10^{-3}$, Figure~\ref{f:unscaled_runtime1} shows the distribution of $K_{\epsilon}(A_n,\vec b)$ for different values of $n$. Each plot has $10^3$ samples.
\begin{figure}[tbp]
\centering
\includegraphics[width=12cm, height=8cm]{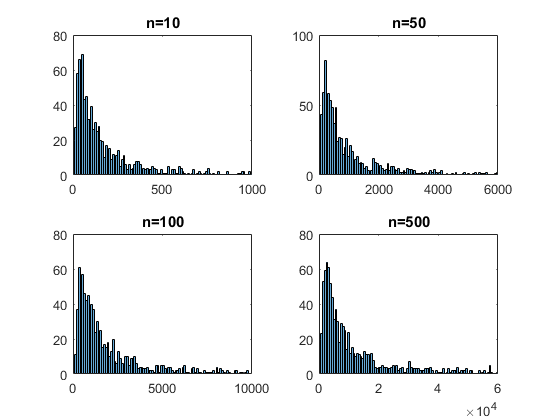}
\caption{Most of the values of $K_{\epsilon}(A_n,\vec b)$ assemble on the left of the plot. When comparing the abscissas of these four plots, we see that as $n$ becomes larger, the range of the values $K_{\epsilon}(A_n,\vec b)$ can achieve also becomes larger.}\label{f:unscaled_runtime1}
\end{figure}
To verify the main theorem, it is equivalent to see if
\begin{align*}
    \frac {n\log (n/\epsilon)}{K_{\epsilon}(A_n,\vec b)}\to X
\end{align*}
in distribution as $n\to \infty$. Figure~\ref{f:scaled_runtime1} shows this convergence. Each plot has $10^3$ samples. In fact, if we decompose $\left(n\log (n/\epsilon)\right)\Big/K_{\epsilon}(A_n,\vec b)$, we find that it involves a term that impedes the speed of convergence. Thus, the convergence in Figure~\ref{f:scaled_runtime1} is quite slow. However, we can improve the speed of convergence. See Appendix~\ref{a:1}.
\begin{figure}[tbp]
\centering
\includegraphics[width=11.5cm, height=7.5cm]{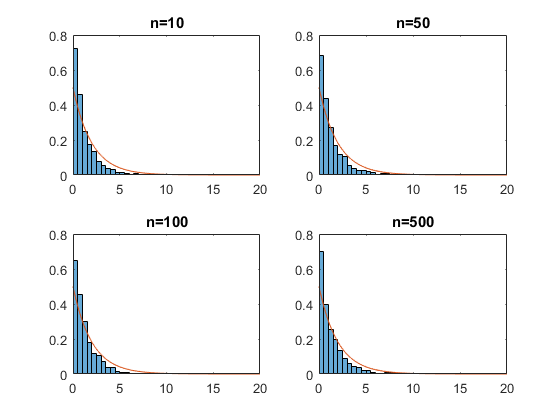}
\caption{The distribution of $\left(n\log(n/\epsilon)\right)\Big/K_{\epsilon}(A_n,\vec b)$ becomes flatter as $n$ becomes larger, which converges to the probability density function for $X\sim \mathrm{exp}(1/2)$. This is a verification of Theorem~\ref{t:main}.}\label{f:scaled_runtime1}
\end{figure}
\begin{remark}
Note that if $K_{\epsilon}(A_n,\vec b)$ is replaced by $k_\epsilon(A_n,\vec b)$, Figure~\ref{f:scaled_runtime1} seems to match better for small values of $n$. However, for $n=10^3$ or larger, the distribution of the actual number of iterations has a heavier tail than the limiting distribution $X\sim \mathrm{exp}(1/2)$. This also appears to hold for the next example.
\end{remark}
\subsubsection{Jacobi unitary ensemble}\label{ex:2} The following definition can be found in \cite[pg.~111]{Forrester10log-gasesand}.
\begin{definition}\label{d:JE}
The Jacobi ensembles are defined as the family of eigenvalue probability density functions proportional to
\begin{align*}
    \prod_{j=1}^{N}(1-\lambda_{j})^{a\beta/2}(1+\lambda_{j})^{b\beta/2}\prod_{1\leq j< k\leq N} \left|\lambda_{k}-\lambda_{j}\right|^{\beta},\quad \lambda_{j}\in [-1,1],
\end{align*}
where $\lambda_{j}$'s are interpreted as eigenvalues, $a,b,N$ are positive integers and $\beta=1,2,$ or $4$. When $\beta=2$ these are referred to as the Jacobi unitary ensembles.
\end{definition}
Let $A=v^{*}v, B=w^{*}w$, where $v$ and $w$ are $n_{1}\times n$ and $n_{2}\times n$ random matrices with entries that are independent and identically distributed standard complex normal random variables. By Proposition $3.6.1$ in \cite[pg.~111]{Forrester10log-gasesand} and Definition~\ref{d:JE}, we know that the eigenvalues $x_{1},x_{2},\cdots,x_{n}$ of the matrix $V=(A+B)^{-1/2}A(A+B)^{-1/2}$ have the joint density function proportional to the probability density function presented in Definition~\ref{d:JE} with
\begin{align*}
N=n,\quad \lambda_{j}=1-2 x_{j},\quad a=n_{1}-n,\quad b=n_{2}-n,\quad \beta=2.
\end{align*}
To verify Theorem~\ref{t:main}, we will therefore focus on $W_n=I-2V$, where $V = (A+B)^{-1/2}A(A+B)^{-1/2}$. We can express the eigenvalue correlations of $W_n$ near $1$ in terms of the Bessel kernel \cite[pg.~1576]{article2}
\begin{align*}
    \mathbb {J}_{\alpha}(u,v)=\frac {J_{\alpha}(\sqrt{u})\sqrt{v}J'_{\alpha}(\sqrt{v})-J_{\alpha}(\sqrt{v})\sqrt{u}J'_{\alpha}(\sqrt{u})}{2(u-v)},
\end{align*}
where $u,v\geq 0$ and $J_{\alpha}$ is the usual Bessel function of the first kind and of order $\alpha$ \cite{DLMF}. Let $\mathbb {J}_{\alpha,s}$ be the integral operator with kernel $\mathbb {J}_{\alpha}(u,v)$ acting on $L^{2}(0,s)$. Then by Corollary 1.2 in \cite[pg.~1578]{article2}, for $s>0$, we have
\begin{align*}
    P_{n}\left(1-\frac {s}{2n^{2}},1\right)\to \text{det}\left(I-\mathbb {J}_{\alpha,s}\right)
\end{align*}
as $n\to \infty$, where $P_{n}(a,b)$ is the probability that there are no eigenvalues in the interval $(a,b)\subset (-1,1)$ and $\text{det}\left(I-\mathbb {J}_{\alpha,s}\right)$ is the Fredholm determinant (see, for example, \cite{article}). By definition, we have
\begin{align*}
    P_{n}\left(1-\frac {s}{2n^{2}},1\right)=\mathbb {P}\left(\lambda_{n}\leq 1-\frac {s}{2n^{2}}\right)=\mathbb {P}\left(n^{2}(1-\lambda_{n})\geq \frac {s}{2}\right).
\end{align*}
Therefore,
\begin{align*}
    \mathbb {P}\left(n^{2}(1-\lambda_{n})\leq \frac {s}{2}\right)\to 1-\text{det}\left(I-\mathbb {J}_{\alpha,s}\right)
\end{align*}
as $n\to \infty$. Let $t=s/2$, we can rewrite it as
\begin{align*}
    \mathbb {P}\left(n^{2}(1-\lambda_{n})\leq t\right)\to 1-\text{det}\left(I-\mathbb {J}_{\alpha,2t}\right)
\end{align*}
as $n\to \infty$. Similarly, we also have
\begin{align*}
    \mathbb {P}\left(n^{2}(\lambda_{1}+1)\leq t\right)\to 1-\text{det}\left(I-\mathbb {J}_{\alpha,2t}\right)
\end{align*}
as $n\to \infty$. Therefore, the assumptions of Theorem~\ref{t:main} are satisfied with $\alpha=\beta=2$. Let $n_{1}=n_{2}=n+2$. Figure~\ref{f:global} shows the global eigenvalue distribution of $W_n$ for different values of $n$. Each plot has $10^3$ samples.
\begin{figure}[tbp]
\centering
\includegraphics[width=11.5cm, height=7.5cm]{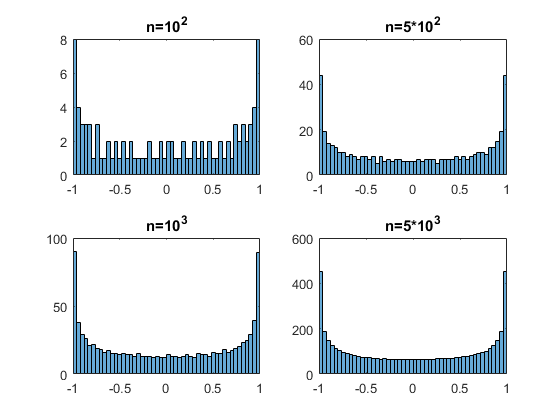}
\caption{We can see that most of the eigenvalues assemble near $-1$ and $1$ and within a small neighbourhood around $0$, the distribution of the eigenvalues is almost uniform.}\label{f:global}
\end{figure}
Consider a quadrature rule \cite[pg.~175]{Gautschi:2011:NA:2208098}
\begin{align*}
    \int_{-1}^{1} f(x)\, dx\approx\sum_{j=1}^{n}f(x_{j})w_{j},
\end{align*}
where $w_{j}$'s are discrete weights and $x_j$'s are the nodes. By a linear transformation, we find
\begin{equation}\label{eq:1}
    \int_{0}^{2t} f(x)\, dx\approx\sum_{j=1}^{n}f\left(t(1+x_{j})\right)\left(tw_{j}\right).
\end{equation}
We use Gauss-Legendre quadrature \cite{Gautschi:2011:NA:2208098}.  We calculate the Bessel kernel $\mathbb J_\alpha(u,v)$ by applying
\begin{align*}
    J'_{\alpha}(x)=J_{\alpha-1}(x)-\frac {\alpha}{x}J_{\alpha}(x)
\end{align*}
and
\begin{align*}
    J''_{\alpha}(x)=J'_{\alpha-1}(x)-\frac {\alpha}{x}J'_{\alpha}(x)+\frac {\alpha}{x^2}J_{\alpha}(x)
\end{align*}
where the last formula is required when $u=v$. Then using the algorithm for calculating the Fredholm determinant from \cite[pg.~874]{article} along with~\eqref{eq:1}, we evaluate the Fredholm determinant and compute the cumulative distribution function $1-\text{det}\left(I-\mathbb {J}_{\alpha,2t}\right)$.
Figure~\ref{f:scaledleig} shows the distribution of $n^{2}(1-\lambda_{n})$ for different values of $n$ and the probability density function of $1-\text{det}\left(I-\mathbb {J}_{\alpha,2t}\right)$, found using a central difference. Each plot has $10^3$ samples. The case for $n^2(\lambda_{1}+1)$ is similar.
\begin{figure}[tbp]
\centering
\includegraphics[width=11.5cm, height=7.5cm]{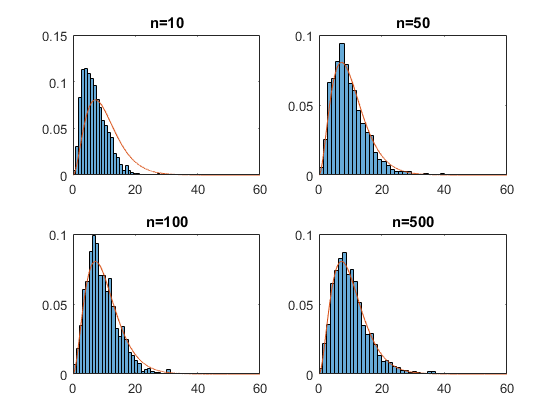}
\caption{As $n$ becomes larger, the distribution of $n^{2}(1-\lambda_{n})$ converges to the probability density function of $1-\text{det}\left(I-\mathbb {J}_{\alpha,2t}\right)$. This is a verification of Corollary 1.2 in \cite[pg.~1578]{article2}.}\label{f:scaledleig}
\end{figure}
Now, we are ready to plot the distribution of $K_{\epsilon}(W_n,\vec b)$. Fix $\epsilon=10^{-3}$, Figure~\ref{f:unruntime} shows the distribution of $K_{\epsilon}(W_n,\vec b)$ for different values of $n$. Each plot has $10^3$ samples.
\begin{figure}[tbp]
\centering
\includegraphics[width=11.5cm, height=7.5cm]{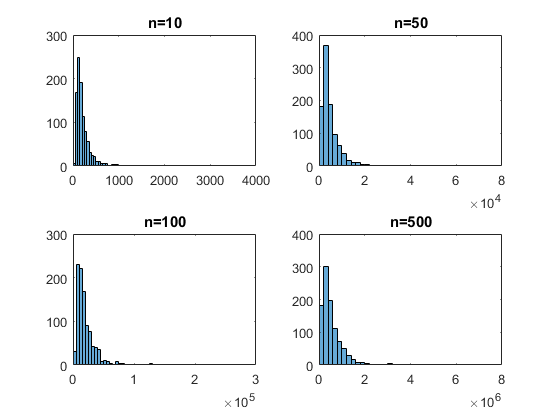}
\caption{Like Figure~\ref{f:unscaled_runtime1}, the values of $K_{\epsilon}(W_n,\vec b)$ assemble on the left of the plot and as $n$ becomes larger, the range of the values $K_{\epsilon}(W_n,\vec b)$ can achieve becomes larger.}\label{f:unruntime}
\end{figure}
Like the previous example, we want to see if $\left(2n^{2}\log(n/\sqrt{\epsilon})\right)\Big/ K_{\epsilon}(W_n,\vec b)$ converges in distribution as $n\to \infty$. Figure~\ref{f:scaledtime} shows this convergence. Each plot has $10^3$ samples.
\begin{figure}[tbp]
\centering
\includegraphics[width=11.5cm, height=7.5cm]{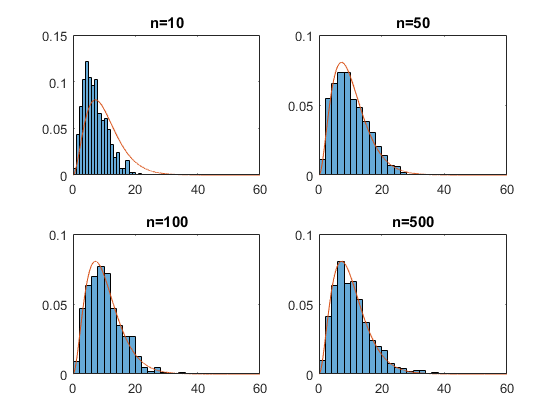}
\caption{As $n$ increases, the distribution of $\left(2n^{2}\log(n/\sqrt{\epsilon})\right)\Big/ K_{\epsilon}(W_n,\vec b)$ converges to the probability density function of $1-\text{det}\left(I-\mathbb {J}_{\alpha,2t}\right)$.}\label{f:scaledtime}
\end{figure}
\section{Lemmas and the proof of Theorem~\ref{t:main}}
In this section, we will first prove some lemmas and then prove the main theorem based on these lemma.
\begin{lemma}\label{lem:1}
Suppose $M_{n}$ is a $n\times n$ symmetric matrix with eigenvalues
\begin{align*}
    -1<\lambda_{1}\leq \lambda_{2}\leq \cdots\leq \lambda_{n}<1.
\end{align*}
Then
\begin{align*}
    \left\|(I-M_{n})^{-1}-\sum_{i=0}^{k} M^{i}_{n}\right\|=\max \left\{\frac {|\lambda_{1}|^{k}}{|1-\lambda_{1}|},\frac {|\lambda_{n}|^{k}}{|1-\lambda_{n}|}\right\}.
\end{align*}
\end{lemma}
\begin{proof}
By Lemma~\ref{lem:neumann}, 
\begin{align*}
    \left\|(I-M_{n})^{-1}-\sum_{i=0}^{k} M^{i}_{n}\right\|=\left\|\sum_{i=k}^{\infty} M^{i}_{n}\right\|=\left\|M^{k}_{n}\sum_{i=0}^{\infty} M^{i}_{n}\right\|.
\end{align*}
We decompose $M_{n}$ as $M_{n}=U\Lambda U^{T}$, where $\Lambda$ is a diagonal matrix formed from the eigenvalues of $M_{n}$ and $U$ is a unitary matrix. Thus,
\begin{align*}
    \left\|M^{k}_{n}\sum_{i=0}^{\infty} M^{i}_{n}\right\|=\left\|\begin{bmatrix} 
    \frac {\lambda^{k}_{1}}{1-\lambda_{1}} & \dots & 0 \\
    \vdots & \ddots & \vdots\\
    0 & \dots & \frac {\lambda^{k}_{n}}{1-\lambda_{n}} 
    \end{bmatrix}\right\|=\max \left\{\frac {|\lambda_{1}|^{k}}{|1-\lambda_{1}|},\cdots,\frac {|\lambda_{n}|^{k}}{|1-\lambda_{n}|}\right\}.
\end{align*}
Suppose $-1<\lambda_{1}\leq  \cdots \leq \lambda_{m-1}\leq 0\leq \lambda_{m}\leq\cdots\leq \lambda_{n}<1$. Consider $\lambda_{i}$, where $i=m,\cdots,n$. Since such $\lambda_{i}$'s are all positive and less than 1, we have $ |\lambda_{i}|^{k}/|1-\lambda_{i}|= \lambda^{k}_{i}/\left(1-\lambda_{i}\right)$. Note that $ \lambda^{k}_{i}/\left(1-\lambda_{i}\right)$ is a strictly increasing function of $\lambda_{i}\in (0,1)$ for any positive integer $k$, we have
\begin{align*}
    \max \left\{\frac {|\lambda_{m}|^{k}}{|1-\lambda_{m}|},\cdots,\frac {|\lambda_{n}|^{k}}{|1-\lambda_{n}|}\right\}=\frac {|\lambda_{n}|^{k}}{|1-\lambda_{n}|}.
\end{align*}
Now, consider $\lambda_{i}$, where $i=1,\cdots,m-1$. We want to show that $ |\lambda_{i}|^{k}/|1-\lambda_{i}|$ is a strictly decreasing function of $\lambda_{i}\in (-1,0)$. Note that it is equivalent to show that $ \lambda^{k}_{i}/\left(1+\lambda_{i}\right)$ is a strictly increasing function of $\lambda_{i}\in (0,1)$. Let $g(\lambda_{i})=\lambda^{k}_{i}/\left(1+\lambda_{i}\right)$ and compute
\begin{align*}
    g'(\lambda_{i})&=\frac {k\lambda^{k-1}_{i}(1+\lambda_{i})-\lambda^{k}_{i}}{(1+\lambda_{i})^{2}}\\
    &=\frac {k\lambda^{k}_{i}-\lambda^{k}_{i}+k\lambda^{k-1}_{i}}{(1+\lambda_{i})^{2}}>0.
\end{align*}
Therefore, $|\lambda_{i}|^{k}/|1-\lambda_{i}|$ is a strictly decreasing function of $\lambda_{i}\in (-1,0)$, and we have
\begin{align*}
    \max \left\{\frac {|\lambda_{1}|^{k}}{|1-\lambda_{1}|},\cdots,\frac {|\lambda_{m-1}|^{k}}{|1-\lambda_{m-1}|}\right\}=\frac {|\lambda_{1}|^{k}}{|1-\lambda_{1}|}.
\end{align*}
Therefore, we conclude
\begin{align*}
    \left\|(I-M_{n})^{-1}-\sum_{i=0}^{k} M^{i}_{n}\right\|=\max \left\{\frac {|\lambda_{1}|^{k}}{|1-\lambda_{1}|},\frac {|\lambda_{n}|^{k}}{|1-\lambda_{n}|}\right\}.
\end{align*}
\end{proof}
\begin{lemma}\label{lem:2}
Let\footnote{$\mathbb {R}_{+}=[0,\infty)$.} $f,g:\mathbb {R}_{+}\to \mathbb {R}_{+}$ be strictly decreasing continuous functions such that their graphs intersect at most once. Define
\begin{align*}
    h(x)=\max \left\{f(x),g(x)\right\}.
\end{align*}
Then
\begin{align*}
    h^{-1}(y)&=\max \left\{f^{-1}(y),g^{-1}(y)\right\},\\
    &\min_{x\in \mathbb {R}^{+}}\max\{f(x),g(x)\}\leq y \leq \max_{x\in \mathbb {R}^{+}}\max\{f(x),g(x)\}.
\end{align*}
\end{lemma}
\begin{proof}
Consider the case where their graphs intersect once. Without loss of generality, assume $f(x^{*})=g(x^{*})$, $x^{*}>0$, and when $x\leq x^{*}$, $g(x)\geq f(x)$; when $x\geq x^{*}$, $f(x)\geq g(x)$. Therefore, by the definition of $h(x)$, we have
\begin{align*}
    h(x)=\begin{cases} g(x), & x\leq x^{*}\\ f(x), & x\geq x^{*} \end{cases}
\end{align*}
As a result, we get
\begin{align*}
    h^{-1}(y)=\begin{cases} g^{-1}(y), & y\geq f(x^{*})\\ f^{-1}(y), & y\leq f(x^{*}) \end{cases}
\end{align*}
Since both $f,g$ are strictly decreasing functions, $f^{-1},g^{-1}$ are also strictly decreasing functions. If $x\leq x^{*}$, then $f(x)\leq g(x)$. Therefore, $g^{-1}(f(x))\geq x=f^{-1}(f(x))$. In other words, if $y=f(x)$ and $x\leq x^{*}$, we have $y=f(x)\geq f(x^{*})$ and $f^{-1}(y)\leq g^{-1}(y)$. Similarly, if $x\geq x^{*}$, then $f(x)\geq g(x)$. Therefore, $g^{-1}(f(x))\leq x=f^{-1}(f(x))$. In other words, if $y=f(x)$ and $x\geq x^{*}$, we have $y=f(x)\leq f(x^{*})$ and $g^{-1}(y)\leq f^{-1}(y)$. Thus, we have
\begin{align*}
    h^{-1}(y)=\max \left\{f^{-1}(y),g^{-1}(y)\right\}.
\end{align*}
If their graphs do not intersect, without loss of generality, assume $f(x)>g(x)$, $x>0$. Then
\begin{align*}
    h(x)=f(x)=\max \left\{f(x),g(x)\right\}.
\end{align*}
Therefore, $f^{-1}(g(x))>x=g^{-1}(g(x))$. In other words, if $y=g(x)$, we have $f^{-1}(y)>g^{-1}(y)$. Thus,
\begin{align*}
    h^{-1}(y)=f^{-1}(y)=\max \left\{f^{-1}(y),g^{-1}(y)\right\}.
\end{align*}
\end{proof}
Based on Lemma~\ref{lem:1}, to find an expression for $K_{\epsilon}$, set both $|\lambda_{1}|^{k}/|1-\lambda_{1}|$ and $|\lambda_{n}|^{k}/|1-\lambda_{n}|$ equal to $\epsilon$.
\begin{definition}\label{d:k}
Define
\begin{align*}
    k_{1}&=\frac {\log{\epsilon}+\log{|1-\lambda_{1}|}}{\log{|\lambda_{1}|}},\\
    k_{n}&=\frac {\log{\epsilon}+\log{|1-\lambda_{n}|}}{\log{|\lambda_{n}|}}.
\end{align*}
\end{definition}
Based on Lemma~\ref{lem:2},
\begin{align*}
    K_{\epsilon}(M_{n})&=\max \left\{\frac {\log{\epsilon}+\log{|1-\lambda_{1}|}}{\log{|\lambda_{1}|}},\frac {\log{\epsilon}+\log{|1-\lambda_{n}|}}{\log{|\lambda_{n}|}}\right\}+\sigma\\
    &=\max \left\{k_{1},k_{n}\right\}+\sigma,
\end{align*}
\DeclarePairedDelimiter{\ceil}{\lceil}{\rceil}
where\footnote{$\ceil[\big]{\cdot}$ here denotes the ceiling function.} $\sigma=\ceil[\big]{\max\{k_{1},k_{n}\}}-\max \{k_{1},k_{n}\}$. 
\begin{definition}
A sequence $(X_n)_{n \geq 0}$ of random variables converge to zero in probability if for every $\epsilon> 0$
\begin{align*}
    \lim_{n \to \infty} \mathbb P(|X_n| > \epsilon ) = 0.
\end{align*}
\end{definition}
\begin{lemma}\label{lem:3}
Suppose a sequence $(X_n)_{n \geq 0}$ of random variables converge in distribution to a random variable $X$. Suppose further that another sequence $(Y_n)_{n \geq 0}$ of random variables converge in probability to 0. Then $X_{n}Y_{n}$ converges in probability to 0.
\end{lemma}
\begin{proof}
We want to show that $\lim_{n\to \infty}\mathbb {P}(|X_{n}Y_{n}|> \epsilon)=0$ for every $\epsilon>0$. The cumulative distribution function of $X$ has an, at most, countable number of discontinuities. We can choose $(\delta_{k})_{k\geq 1}$, where $\delta_k\in (0,1/k)$ and $k\geq 1$ as the sequence converging to $0$ such that for all $k$, $\epsilon/\delta_{k}$ is a point of continuity for $F_{|X|}$. This is possible since $X/\epsilon$ has an, at most, countable number of discontinuities. We have
\begin{align*}
    \mathbb {P}(|X_{n}Y_{n}|> \epsilon)=\mathbb {P}(|X_{n}Y_{n}|> \epsilon,|Y_{n}|\leq \delta_k)+\mathbb {P}(|X_{n}Y_{n}|> \epsilon,|Y_{n}|>\delta_k).
\end{align*}
Given that $|X_{n}Y_{n}|> \epsilon$ and $|Y_{n}|\leq \delta_k$, we have $|X_{n}|>\epsilon/\delta_k$. Since $\mathbb {P}(|X_{n}Y_{n}|> \epsilon,|Y_{n}|>\delta_k)\leq \mathbb {P}(|Y_{n}|>\delta_k)$, we have
\begin{align*}
    \mathbb {P}(|X_{n}Y_{n}|> \epsilon)\leq \mathbb {P}\left(|X_{n}|> \frac {\epsilon}{\delta_k}\right)+\mathbb {P}(|Y_{n}|>\delta_k).
\end{align*}
Since $\epsilon/\delta_k$ is a point of continuity we know that $\lim_{n\to \infty}\mathbb {P}(|Y_{n}|>\delta_k)=0$, we have for all $k$
\begin{align*}
    \limsup_{n\to \infty}\mathbb {P}(|X_{n}Y_{n}|>\epsilon)\leq \limsup_{n\to \infty}\mathbb {P}\left(|X_{n}|>\frac {\epsilon}{\delta_k}\right)=\mathbb {P}\left(|X|>\frac {\epsilon}{\delta_k}\right).
\end{align*}
By taking $\delta_k \to 0^{+}$ along $(\delta_k)_{k\geq 1}$, we have
\begin{align*}
    \lim_{n\to \infty}\mathbb {P}(|X_{n}Y_{n}|>\epsilon)=0.
\end{align*}
\end{proof}
The following lemma is from \cite[pg.~105]{Durrett2010}.
\begin{lemma}[Converging together lemma]\label{lem:4}
Suppose a sequence $(X_n)_{n \geq 0}$ of random variables converge in distribution to a random variable $X$ as $n\to \infty$.  Suppose further that there is another sequence $(Y_n)_{n \geq 0}$ of random variables such that $Y_n - X_n$ converges to zero in probability as $n\to \infty$.  Then $Y_{n}$ converges to $X$ in distribution as $n\to \infty$.
\end{lemma}
\begin{proof}
Let $F_{X_n}$ be the cumulative distribution function of $X_{n}$ and $F_X$ the cumulative distribution function of $X$. Let $x$ be a continuity point of $F_X$ and $\epsilon>0$. For the upper bound on $Y_{n}$,
\begin{align*}
    \mathbb {P}(Y_{n}\leq x)&=\mathbb {P}\left(Y_{n}\leq x,\left|Y_{n}-X_{n}\right|\leq \epsilon\right)+\mathbb {P}\left(Y_{n}\leq x,\left|Y_{n}-X_{n}\right|> \epsilon\right)\\
    &\leq \mathbb {P}\left(X_{n}\leq x+\epsilon\right)+\mathbb {P}\left(\left|Y_{n}-X_{n}\right|> \epsilon\right).
\end{align*}
Since $\lim_{n\to \infty}\mathbb {P}\left(\left|Y_{n}-X_{n}\right|> \epsilon\right)=0$ and if $x+\epsilon$ is a continuity point of $F_X$, 
\begin{align*}
    \mathbb {P}\left(X_{n}\leq x+\epsilon\right)=F_{X_n}\left(x+\epsilon\right)\to F_X\left(x+\epsilon\right)
\end{align*}
in distribution as $n\to \infty$. Therefore,
\begin{align*}
    \limsup_{n\to \infty}\mathbb {P}\left(Y_{n}\leq x\right)\leq F_X(x+\epsilon)
\end{align*}
for all $\epsilon$ such that $x+\epsilon$ is a continuity point of $F_X$. Such an $\epsilon$ exists since $F_X$ has an, at most, countable number of discontinuities. And indeed there must exist a sequence of choices of $\epsilon>0$ such that $\epsilon\to 0^+$ along this sequence. Take $\epsilon \to 0^{+}$ along this sequence,
\begin{align*}
    \limsup_{n\to \infty}\mathbb {P}\left(Y_{n}\leq x\right)\leq F_X(x).
\end{align*}
For the lower bound on $Y_n$,
\begin{align*}
    \mathbb {P}\left(X_{n}\leq x-\epsilon\right)&=\mathbb {P}\left(X_{n}\leq x-\epsilon,\left|Y_{n}-X_{n}\right|\leq \epsilon\right)+\mathbb {P}\left(X_{n}\leq x-\epsilon,\left|Y_{n}-X_{n}\right|> \epsilon\right)\\
    &\leq \mathbb {P}\left(Y_{n}\leq \epsilon\right)+\mathbb {P}\left(\left|Y_{n}-X_{n}\right|> \epsilon\right).
\end{align*}
If $x-\epsilon$ is a continuity point of $F_X$, 
\begin{align*}
    \liminf_{n\to \infty}\mathbb {P}\left(Y_{n}\leq \epsilon\right)\geq F_X(x-\epsilon).
\end{align*}
for all $\epsilon$ such that $x-\epsilon$ is a continuity point of $F_X$. Such $\epsilon$ exists since that $F_X$ has an, at most, countable number of discontinuities. And indeed there must exist a sequence of choices of $\epsilon>0$ such that $\epsilon\to 0^+$ along this sequence. Again, take $\epsilon\to 0^{+}$ along this sequence,
\begin{align*}
    \liminf_{n\to \infty}\mathbb {P}\left(Y_{n}\leq \epsilon\right)\geq F_X(x).
\end{align*}
These two bounds imply
\begin{align*}
    \lim_{n\to \infty}\mathbb {P}\left(Y_{n}\leq x\right)=F_X(x).
\end{align*}
\end{proof}
\begin{lemma}\label{lem:5}
Suppose a sequence $(X_n)_{n \geq 0}$ of random variables converge in distribution to a random variable $X$ as $n\to \infty$ and $X,X_{n}>0$ almost surely. Set $\rho_n = 1 - X_n/n^{\alpha}$ and $\alpha>0$. Then
\begin{align*}
    \frac{1}{n^{\alpha} \log |\rho_n|}
\end{align*}
converges in distribution to $-1/X$ as $n\to \infty$.
\end{lemma}
\begin{proof}
By Lemma~\ref{lem:4}, it is equivalent to show that $1/n^{\alpha} \log |\rho_n|$ converges in probability to $-1/X_{n}$. Let 
\begin{align*}
    Y_{n}=\left|\frac{1}{n^{\alpha} \log |\rho_n|}+\frac {1}{X_{n}}\right|,
\end{align*}
then we want to show that $\lim_{n \to \infty} \mathbb P(Y_n > \epsilon ) = 0$ for every $\epsilon>0$. Let 
\begin{align*}
    A_{n}=\left\{\left|1-\rho_{n}\right|\leq \frac {1}{2}\right\}=\left\{\left|\frac {X_{n}}{n^{\alpha}}\right|\leq \frac {1}{2}\right\},
\end{align*}
then by Lemma~\ref{lem:3}, $\lim_{n \to \infty} \mathbb {P}(A_{n})=1$. Therefore,
\begin{align*}
    \limsup_{n\to \infty}\mathbb {P}(Y_{n}> \epsilon)&\leq\limsup_{n \to \infty}\mathbb {P}(Y_{n}> \epsilon,A_{n})+\limsup_{n\to \infty}\mathbb {P}(Y_{n}> \epsilon,A^\complement_{n})\\
    &=\limsup_{n \to \infty}\mathbb {P}(Y_{n}> \epsilon,A_{n})
\end{align*}
since $\limsup_{n\to \infty}\mathbb {P}(Y_{n}> \epsilon,A^\complement_{n})\leq \limsup_{n\to \infty}\mathbb {P}(A^\complement_{n})=0$. Thus, we want to show that $\limsup_{n \to \infty}\mathbb {P}(Y_{n}> \epsilon,A_{n})=0$. Let $-1/2\leq x=X_{n}/n^{\alpha}\leq 1/2$, then $\log|1-x|=\log(1-x)$. By Taylor's Theorem,
\begin{align*}
    \log(1-x)=-x-\frac {1}{(1-\xi)^{2}}x^{2},
\end{align*}
where $\xi \in [-\frac {1}{2},\frac {1}{2}]$. Since $f(\xi)=1/(1-\xi)^{2}$ is a strictly increasing function of $\xi$, we have
\begin{align*}
     \left|\log(1-x)+x\right|\leq 4x^{2}.
\end{align*}
Using $x=X_{n}/n^{\alpha}$ and multiplying both sides by $n^{\alpha}$, we find
\begin{align*}
    \left|n^{\alpha}\log \rho_{n}+ X_{n}\right|\leq 4\left(\frac {X^{2}_{n}}{n^{\alpha}}\right).
\end{align*}
When $1/2< x\leq 3/2$, $|\log x|\geq |x-1|/2$. Then $|\log \rho_{n}|\geq |\rho_{n}-1|/2$ on $A_{n}$ since $1/2\leq \rho_{n}\leq 3/2$, and we have 
\begin{align*}
    \left|n^{\alpha}(\log \rho_{n})X_{n}\right|\geq \frac {X^{2}_{n}}{2}.
\end{align*}
Therefore on $A_{n}$,
\begin{align*}
    Y_{n}=\left|\frac{1}{n^{\alpha} \log |\rho_n|}+\frac {1}{X_{n}}\right|=\left|\frac {n^{\alpha}\log \rho_{n}+X_{n}}{n^{\alpha}(\log \rho_{n})X_{n}}\right|\leq \frac {4X^{2}_{n}/n^{\alpha}}{X^{2}_{n}/2}=\frac {8}{n^{\alpha}},
\end{align*}
and
\begin{align*}
    \limsup_{n \to \infty}\mathbb {P}(Y_{n}> \epsilon,A_{n})\leq \limsup_{n\to \infty}\mathbb {P}\left(\frac {8}{n^{\alpha}}>\epsilon,A_{n}\right)=0.
\end{align*}
Thus, we conclude that
\begin{align*}
    \frac {1}{n^{\alpha}\log |\rho_{n}|}\to -\frac {1}{X}
\end{align*}
in distribution as $n\to \infty$.
\end{proof}
\begin{proof}[Proof of Theorem~\ref{t:main}]
Recall Definition~\ref{d:k}:
\begin{align*}
    k_{n}=\frac {\log{\epsilon}+\log{|1-\lambda_{n}|}}{\log{|\lambda_{n}|}}.
\end{align*}
Given that $\lambda_{n}=1-\xi_{n}/n^{\alpha}$, where $\xi_{n}>0$ and $\xi_{n}\to X$ in distribution as $n\to \infty$, we have
\begin{align*}
    k_{n}=\frac {-\alpha \log n+\log \epsilon \xi_{n}}{\log \left|1-\frac {\xi_{n}}{n^{\alpha}}\right|}=\frac {-\alpha \log (n/\epsilon^{1/\alpha})}{\log \left|1-\frac {\xi_{n}}{n^{\alpha}}\right|}+\frac {\log \xi_{n}}{\log \left|1-\frac {\xi_{n}}{n^{\alpha}}\right|}.
\end{align*}
Let
\begin{align*}
    \widetilde {k}_{n}=\frac {-\alpha \log (n/\epsilon^{1/\alpha})}{\log \left|1-\frac {\xi_{n}}{n^{\alpha}}\right|}.
\end{align*}
By Lemma~\ref{lem:5}, we know that
\begin{align*}
    \frac {\widetilde {k}_{n}}{\alpha \log (n/\epsilon^{1/\alpha})n^{\alpha}}=-\frac {1}{n^{\alpha}\log \left|1-\frac {\xi_{n}}{n^{\alpha}}\right|}\to \frac {1}{X}
\end{align*}
in distribution as $n\to \infty$. Moreover, by Lemma~\ref{lem:3},
\begin{align*}
    \frac {\log \xi_{n}}{\alpha \log (n/\epsilon^{1/\alpha})}\to 0
\end{align*}
in probability as $n\to \infty$, and therefore,
\begin{align*}
    \frac {\widetilde {k}_{n}-k_{n}}{\alpha \log (n/\epsilon^{1/\alpha})n^{\alpha}}=\frac {\log \xi_{n}}{\alpha \log (n/\epsilon^{1/\alpha})}\cdot \left(-\frac {1}{n^{\alpha}\log \left|1-\frac {\xi_{n}}{n^{\alpha}}\right|}\right)\to 0
\end{align*}
in probability as $n\to \infty$ by Lemma~\ref{lem:3}. Finally, by Lemma~\ref{lem:4}, we find
\begin{align*}
    \frac {k_{n}}{\alpha \log (n/\epsilon^{1/\alpha})n^{\alpha}}\to \frac {1}{X}
\end{align*}
in distribution as $n\to \infty$. Similarly, in Definition~\ref{d:k}, recall
\begin{align*}
    k_{1}=\frac {\log{\epsilon}+\log{|1-\lambda_{1}|}}{\log{|\lambda_{1}|}}.
\end{align*}
Given that $\lambda_{1}=-1+\xi_{1}/n^{\beta}$, where $\xi_{1}>0$ and $\xi_{1}\to Y$ in distribution as $n\to \infty$, we write
\begin{align*}
    k_{1}=\frac {\log \left(2-\frac {\xi_{1}}{n^{\beta}}\right)\epsilon}{\log \left|1-\frac {\xi_{1}}{n^{\beta}}\right|}.
\end{align*}
Let $\zeta_{1}=2-\xi_{1}/n^{\beta}$, then $\zeta_{1}\to 2$ in probability. Thus, by Lemma~\ref{lem:3}, we have
\begin{align*}
    \frac {\log \zeta_{1}\epsilon}{\log (n/\epsilon^{1/\beta})}\to 0
\end{align*}
in probability as $n\to \infty$. By Lemma~\ref{lem:5}, we have
\begin{align*}
    \frac {1}{n^{\beta}\log \left|1-\frac {\xi_{n}}{n^{\beta}}\right|}\to -\frac {1}{X}
\end{align*}
in distribution as $n\to \infty$. Therefore,
\begin{align*}
    \frac {k_{1}}{\log (n/\epsilon^{1/\beta})n^{\beta}}&=\frac {\log \zeta_{1}\epsilon}{\log (n/\epsilon^{1/\beta})n^{\beta}\log \left|1-\frac {\xi_{n}}{n^{\beta}}\right|}\\
    &=\left(\frac {\log \zeta_{1}\epsilon}{\log (n/\epsilon^{1/\beta})}\right)\cdot \left(\frac {1}{n^{\beta}\log \left|1-\frac {\xi_{n}}{n^{\beta}}\right|}\right)\to 0
\end{align*}
in probability as $n\to \infty$ by Lemma~\ref{lem:3}. Given that $\alpha\geq \beta>0$, we have
\begin{align*}
    \frac {K_{\epsilon}(M_{n})}{\alpha \log (n/\epsilon^{1/\alpha})n^{\alpha}}=\max \left\{\frac {k_{1}}{\alpha \log (n/\epsilon^{1/\alpha})n^{\alpha}},\frac {k_{n}}{\alpha \log (n/\epsilon^{1/\alpha})n^{\alpha}}\right\}+\frac {\sigma}{\alpha \log (n/\epsilon^{1/\alpha})n^{\alpha}}.
\end{align*}
Let $1/X_{n}:=k_{n}/\alpha \log (n/\epsilon^{1/\alpha})n^{\alpha}\to 1/X$ in distribution as $n\to \infty$ and $Y_{n}:=k_{1}/\alpha \log (n/\epsilon^{1/\alpha})n^{\alpha}\to 0$ in probability as $n\to \infty$. Fix $0<\epsilon<1/2$ and define
\begin{align*}
    A_{n,\epsilon}=\left\{Y_{n}>\epsilon\right\}.
\end{align*}
We have
\begin{align*}
    \mathbb {P}\left(Y_{n}>\frac {1}{X_{n}}\right)=\mathbb {P}\left(Y_{n}>\frac {1}{X_{n}},A_{n,\epsilon}\right)+\mathbb {P}\left(Y_{n}>\frac {1}{X_{n}},A^{\complement}_{n,\epsilon}\right).
\end{align*}
Applying $\limsup_{n\to \infty}$ we find
\begin{align*}
    \limsup_{n\to \infty}\mathbb {P}\left(Y_{n}>\frac {1}{X_{n}}\right)\leq \limsup_{n\to \infty}\mathbb {P}\left(Y_{n}>\frac {1}{X_{n}},A_{n,\epsilon}\right)+\limsup_{n\to \infty}\mathbb {P}\left(Y_{n}>\frac {1}{X_{n}},A^{\complement}_{n,\epsilon}\right).
\end{align*}
Then
\begin{align*}
    \limsup_{n\to \infty}\mathbb {P}\left(Y_{n}>\frac {1}{X_{n}},A_{n,\epsilon}\right)\leq \limsup_{n\to \infty}\mathbb {P}\left(A_{n,\epsilon}\right)=0,
\end{align*}
and since $1/X_{n}\to 1/X$ in distribution as $n\to \infty$,
\begin{align*}
    \limsup_{n\to \infty}\mathbb {P}\left(Y_{n}>\frac {1}{X_{n}}\right)&\leq \limsup_{n\to \infty}\mathbb {P}\left(Y_{n}>\frac {1}{X_{n}},A^{\complement}_{n,\epsilon}\right)\\
    &=\limsup_{n\to \infty}\mathbb {P}\left(\frac {1}{X_{n}}<\epsilon\right)\\
    &\leq F_{1/X}(\epsilon).
\end{align*}
 Therefore, applying $\lim_{\epsilon \to 0^{+}}$ on both sides, we get
\begin{align*}
    \lim_{\epsilon \to 0^{+}}\limsup_{n\to \infty}\mathbb {P}\left(Y_{n}>\frac {1}{X_{n}}\right)\leq \lim_{\epsilon \to 0^{+}}F_{1/X}(\epsilon)=0
\end{align*}
since $1/X>0$ and $\lim_{\epsilon \to 0^{+}}F_{1/X}(\epsilon)=\lim_{\epsilon \to 0^{+}}\mathbb {P}\left(1/X\leq \epsilon\right)=0$. Let $Z_{n}=K_{\epsilon}(M_{n})\Big/\left(\alpha \log (n/\epsilon^{1/\alpha})n^{\alpha}\right)$ and
\begin{align*}
    M_{n}=\left\{Y_{n}>\frac {1}{X_{n}}\right\},
\end{align*}
then we have
\begin{align*}
    \mathbb {P}\left(\left|\frac {1}{X_{n}}-Z_{n}\right|\leq \epsilon\right)=\mathbb {P}\left(\left|\frac {1}{X_{n}}-Z_{n}\right|\leq \epsilon, M_{n}\right)+\mathbb {P}\left(\left|\frac {1}{X_{n}}-Z_{n}\right|\leq \epsilon, M^{\complement}_{n}\right).
\end{align*}
Applying $\lim_{\epsilon \to 0^{+}}\limsup_{n\to \infty}$ on both sides, we get
\begin{align*}
    \lim_{\epsilon \to 0^{+}}\limsup_{n\to \infty}\mathbb {P}\left(\left|\frac {1}{X_{n}}-Z_{n}\right|\leq \epsilon\right)&\leq \lim_{\epsilon \to 0^{+}}\limsup_{n\to \infty}\mathbb {P}\left(\left|\frac {1}{X_{n}}-Z_{n}\right|\leq \epsilon,M_{n}\right)\\
    &+\lim_{\epsilon \to 0^{+}}\limsup_{n\to \infty}\mathbb {P}\left(\left|\frac {1}{X_{n}}-Z_{n}\right|\leq \epsilon,M^{\complement}_{n}\right).
\end{align*}
Since
\begin{align*}
    \lim_{\epsilon \to 0^{+}}\limsup_{n\to \infty}\mathbb {P}\left(\left|\frac {1}{X_{n}}-Z_{n}\right|\leq \epsilon,M_{n}\right) \leq \lim_{\epsilon \to 0^{+}}\limsup_{n\to \infty}\mathbb {P}\left(M_{n}\right)=0,
\end{align*}
we get
\begin{align*}
    \lim_{\epsilon \to 0^{+}}\limsup_{n\to \infty}\mathbb {P}\left(\left|\frac {1}{X_{n}}-Z_{n}\right|\leq \epsilon\right)&\leq \lim_{\epsilon \to 0^{+}}\limsup_{n\to \infty}\mathbb {P}\left(\left|\frac {1}{X_{n}}-Z_{n}\right|\leq \epsilon,M^{\complement}_{n}\right)\\
    &\leq \lim_{\epsilon \to 0^{+}}\limsup_{n\to \infty}\mathbb {P}\left(M^{\complement}_{n}\right)\\
    &=1.
\end{align*}
Thus, $Z_{n}$ and $1/X_{n}$ have the same limiting distribution. In other words,
\begin{align*}
    \frac {K_{\epsilon}(M_{n})}{\alpha \log (n/\epsilon^{1/\alpha})n^{\alpha}}\to \frac {1}{X}
\end{align*}
in distribution as $n\to \infty$.
\end{proof}
\appendix
\section{Improvement on speed of convergence}\label{a:1}
Following the definitions of the proof of Theorem~\ref{t:main}, we know that
\begin{align*}
    K_{\epsilon}(A,\vec b)\Big/\left(\alpha \log (n/\epsilon^{1/\alpha})n^{\alpha}\right)\quad\text{and}\quad k_{n}\Big/\left(\alpha \log (n/\epsilon^{1/\alpha})n^{\alpha}\right)
\end{align*}
have the same limiting distribution and recall 
\begin{align*}
    k_{n}=\frac {-\alpha \log n+\log \epsilon \xi_{n}}{\log \left|1-\frac {\xi_{n}}{n^{\alpha}}\right|}=\frac {-\alpha \log (n/\epsilon^{1/\alpha})}{\log \left|1-\frac {\xi_{n}}{n^{\alpha}}\right|}+\frac {\log \xi_{n}}{\log \left|1-\frac {\xi_{n}}{n^{\alpha}}\right|},
\end{align*}
where $\epsilon=10^{-3}$ and $\xi_{n}\to \text{exp}(1/2)$ in distribution as $n\to \infty$ for the example in Section~\ref{ex:1}. By Lemma~\ref{lem:5},
\begin{align*}
    -\frac {1}{n^{\alpha}\log \left|1-\frac {\xi_{n}}{n^{\alpha}}\right|}\to \frac {1}{X}
\end{align*}
in distribution as $n\to \infty$. Therefore, divide $k_{n}$ by $n^{\alpha} \log (n/\epsilon^{1/\alpha})$ and factor out the term $-1\Big/n^{\alpha}\log \left|1-\frac {\xi_{n}}{n^{\alpha}}\right|$ giving
\begin{align*}
    \frac {k_{n}}{ n^{\alpha}\log (n/\epsilon^{1/\alpha})}=-\frac {1}{n^{\alpha}\log \left|1-\frac {\xi_{n}}{n^{\alpha}}\right|}\left(\alpha-\frac {\log \xi_{n}}{\log(n/\epsilon^{1/\alpha})}\right).
\end{align*}
We can treat $\left(\alpha-\log \xi_{n}\Big/\log(n/\epsilon^{1/\alpha})\right)$ as a correction term and replace $\log(\xi_{n})$ with its expectation. To get a faster convergence, we move $\left(\alpha-\mathbb {E}\left[\log \xi_{n}\right]\Big/\log(n/\epsilon^{1/\alpha})\right)$ to the left hand side and take the reciprocal to find
\begin{align*}
    Z=\frac {\alpha\log(n/\epsilon^{1/\alpha})-\mathbb {E}\left[\log(\xi_{n})\right]}{\log(n/\epsilon^{1/\alpha})}\cdot \frac {n^{\alpha}\log(n/\epsilon^{1/\alpha})}{k_{n}}.
\end{align*}
For the example in Section~\ref{ex:1}, we have
\begin{align*}
    Z_1=\frac {\log(n/\epsilon)-\mathbb {E}\left[\log(\xi_{n})\right]}{\log(n/\epsilon)}\cdot \frac {n\log(n/\epsilon)}{k_{n}}.
\end{align*}
Figure~\ref{f:scaled_runtime1refined} shows the refinement. Each plot has $10^3$ samples.
\begin{figure}[tbp]
\centering
\includegraphics[width=11.5cm, height=7.5cm]{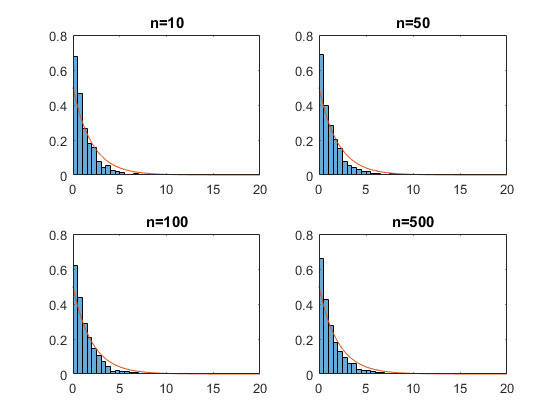}
\caption{The distribution of $Z_1$ converges to the probability density function of $X\sim \text{exp}(1/2)$ as $n$ becomes larger. The figure is much like Figure~\ref{f:scaled_runtime1} but with a faster speed of convergence.}\label{f:scaled_runtime1refined}
\end{figure}
\bibliographystyle{amsalpha}
\bibliography{library}
\end{document}